\documentclass[12pt]{article}

\usepackage[utf8]{inputenc}

\usepackage[T1]{fontenc}
\usepackage{amsmath}
\usepackage{amsfonts}
\usepackage{amssymb}
\usepackage{amsthm}
\usepackage{mathrsfs}

\usepackage[left=20mm,right=20mm]{geometry}
\usepackage[all]{xy}

\theoremstyle{definition}
\newtheorem{defi}{Definition}[section]
\newtheorem{lem}[defi]{Lemma}
\newtheorem*{mteo}{Main Theorem}

\newtheorem{ex}[defi]{Example}
\newtheorem{obs}[defi]{Remark}

\DeclareMathOperator{\rad}{\operatorname{rad}}
\DeclareMathOperator{\md}{\operatorname{mod}}

\DeclareMathOperator{\Hom}{\operatorname{Hom}}

\title{Leaps in the depth of compositions of irreducible morphisms}
\author{Viktor Chust\thanks{Corresponding author: viktorchust.math@gmail.com}, Flávio U. Coelho, \\ Institute of Mathematics and Statistics, University of São Paulo}
\date{}

\begin{document}

\maketitle

\begin{abstract}
In this article, we give a family of examples of algebras, showing that for every $n \geq 2$ and $m  \geq 0$, there is an algebra displaying a path of $n$ irreducible morphisms between indecomposable modules whose composite lies in the $(n+m+3)$-th power of the radical, but not in the $(n+m+4)$-th power. Such an algebra may be also supposed to be string and representation-finite.

\vspace{1ex}


\end{abstract}

\section*{Introduction}

Here we deal with finite-dimensional algebras over an algebraically closed field $k$. Given an algebra $A$, we study the category $\md A$ of finitely generated, right modules over $A$.

Given two indecomposable $A$-modules $X$ and $Y$, we denote by $\rad_A(X,Y)$ the set of non-isomorphisms $X \rightarrow Y$, and we call the elements of $\rad_A(X,Y)$ as \textbf{radical morphisms}. This definition is extended to general modules as follows: $\rad_A(\oplus_{i=1}^n X_i,\oplus_{j=1}^m Y_j) = \oplus_{i=1}^n\oplus_{j=1}^m \rad_A(X_i,Y_j)$. Actually $\rad_A$ is an ideal of the category $\md A$: composing a radical morphism with any other morphism results in a radical morphism. With that, we can consider the \textbf{powers} of the ideal $\rad_A$ - these are defined recursively by: $\rad_A^0 = \Hom_A, \rad_A^1= \rad_A, \rad_A^n = \rad_A^{n-1} \cdot \rad_A$, where the product $\cdot$ stands for composition of morphisms. We also define $\rad_A^{\infty} = \cap_{n \geq 0} \rad_A^n$.

If $X,Y$ are indecomposable, then a morphism $X \rightarrow Y$ is called \textbf{irreducible} if it belongs to $\rad_A(X,Y) \setminus \rad_A^2(X,Y)$. Irreducible morphisms are of key importance, since, as shown by Auslander-Reiten theory, these morphisms generate any other morphism modulo $\rad^{\infty}$.

As stated above, a single irreducible morphism between indecomposable modules belongs to $\rad$, but not $\rad^2$. The composite of $n$ irreducible morphisms belongs to $\rad^n$, but it is not true in general that it does not belong to $\rad^{n+1}$. (See \cite{CCT1} for an example of 2 irreducible morphisms whose composite is non-zero and belongs to $\rad^{\infty}$). A classical result in this theory is the Igusa-Todorov Theorem (\cite{IT1}, \S 13), which says that the composite of $n$ irreducible morphisms along a sectional path belongs to $\rad^n \setminus \rad^{n+1}$ (and in particular is non-zero).

Deciding in which cases there might be a non-zero composition of $n$ irreducible morphisms belonging to $\rad^{n+1}$ has become an interesting line of investigation. For example, \cite{CCT1} fully characterizes the case of 2 irreducible morphisms whose composite belongs to $\rad^3 \setminus \{0\}$.

Then, \cite{AlC} proved that if the composite of two irreducible morphisms belongs to $\rad^3$, it must also belong to $\rad^5$. A few years later, \cite{C} showed that if the composition of 3 irreducible morphisms belongs to $\rad^4$, then it must belong to $\rad^6$. So we see that the depth (i.e., the power of the radical ideal to which a morphism belongs) of compositions of irreducible morphisms can be expected to make `leaps', which in both cases $n=2$ and $n=3$ would be a leap of size 3.

Recently, in \cite{CGS}, Chaio et al. have stated as an open problem that if the composition of $n$ irreducible morphisms between indecomposable modules belongs to $\rad^{n+1}$, then it must also belong to $\rad^{n+3}$, i.e., a leap of size 3 is always to be expected. They have also shown there that for every $n \geq 3$ and every $m \geq 4$, there is a \textit{string algebra} (see definition below) for which there is a path of $n$ irreducible morphisms $f_1,\ldots,f_n$ between indecomposable modules whose composition $f_n \ldots f_1$ belongs to $\rad^{n+m} \setminus \rad^{n+m+1}$, and such that the compositions $f_n \ldots f_2$ and $f_{n-1} \ldots f_1$ do not belong to $\rad^n$. This algebra can be supposed of finite type, i.e., it has a finite number of indecomposable modules up to isomorphism. This theorem thus exhibits a family of algebras, each of which displaying a composition of irreducible morphisms giving leaps of arbitrary size in the power of the radical, as long as this size is greater than 4. 

In this article we wish to contribute to this problem, inspired by this latter result. Here our goal is to give another family of examples in which there are leaps of arbitrary size greater than 3. However, we will need to drop the minimality condition above (i.e., that the composition of $n-1$ consecutive morphisms does not belong to $\rad^n$). The precise statement of our main theorem is as follows:

\begin{mteo}
\label{th:main}
    For $n \geq 2$ and $m \geq 0$, there is an algebra, which we will denote by $A(n,m)$, which is a string algebra of finite type such that there are $n$ irreducible morphisms between indecomposable modules whose composite (is non-zero) and belongs to $\rad^{n+m+3} \setminus \rad^{n+m+4}$.
\end{mteo}

So, the open problem posed in \cite{CGS} states that the depth of the composite of $n$ irreducible morphisms cannot be $n+1$ or $n+2$, and our result here adds to that, by saying that these are the only depths that could be forbidden for composites, examples being existent for any other depths.

This paper is organized the following way: Section~\ref{sec:notations} is devoted to necessary notations, while in Section~\ref{sec:main theorem}, after giving a motivating example, we proceed to the definition of the family of algebras $A(n,m)$ and then to the proof that they satisfy the desired properties.

\section{Notations}
\label{sec:notations}

A \textbf{quiver} $Q$ is a 4-uple $Q=(Q_0,Q_1,s,e)$, where $Q_0$ is a set of \textbf{vertices}, $Q_1$ is a set of \textbf{arrows}, and $s,e:Q_1 \rightarrow Q_0$ are two functions which give, respectively, the \textbf{start} and the \textbf{end} of an arrow. To each arrow $\alpha \in Q_1$ we can consider its formal inverse $\alpha^{-1}$, and we establish by convention that $s(\alpha^{-1}) = e(\alpha)$ and $e(\alpha^{-1}) = s(\alpha)$.

A \textbf{walk} over $Q$ is a sequence $\beta_n \ldots \beta_1$, where, for every $1 \leq i \leq n$, $\beta_i$ is either an arrow or the inverse of an arrow, and where $e(\beta_i) = s(\beta_{i+1})$ for every $1 \leq i < n$. We say that this walk is a \textbf{reduced walk} if additionally $\beta_i \neq \beta_{i+1}^{-1}$ for every $1 \leq i < n$. If $w= \beta_n \ldots \beta_1$ is a walk, we extend the notations used for start and end vertices: $s(w) \doteq s(\beta_1)$ and $e(w) \doteq e(\beta_n)$.

If $\beta_n \ldots \beta_1$ is a walk and all $\beta_i's$ are arrows (rather than inverses of arrows), then we say that this walk is a \textbf{path of length $n$}. Additionally, one associates to each vertex $x$ of $Q$ a trivial \textbf{path of length 0}, denoted by $\epsilon_x$. Of course, $s(\epsilon_x)=e(\epsilon_x) =x$.

Given a quiver $Q$, we can consider the $k$-vector space $kQ$ generated by all paths over $Q$. This space can be endowed with a product operation by defining the product of two paths by their concatenation and then extending linearly. This yields a $k$-algebra structure over $kQ$, and we call this the \textbf{path algebra} over $Q$.

Moreover, one can also consider an ideal $I$ of $kQ$, generated by a set of \textbf{relations} over $Q$. Usually, a relation is a $k$-linear combination of paths of length at least two, sharing the same start and end vertex. The quotient $kQ/I$ is said to be the path algebra over $Q$, \textbf{bound by} the relations in $I$.

Given a quiver $Q$, a \textbf{representation} $M$ over $Q$ is given by the assignment of a finite-dimensional $k$-vector space $M_x$ to each vertex $x \in Q_0$, and of a $k$-linear map $M_{\alpha}:M_{s(\alpha)} \rightarrow M_{e(\alpha)}$ to each arrow $\alpha \in Q_1$. If $A = kQ/I$ is a bound path algebra, it is well-known that the representations over $Q$ that vanish over the relations in $I$ form a category, which is equivalent to the category of finitely-generated $A$-modules.

If $A = kQ/I$ is a bound path algebra, we will denote by $P(i),I(i) \text{, and } S(i)$ respectively, the indecomposable projective, indecomposable injective and simple representations associated with vertex $i$.

For unexplained concepts of module theory and representation theory, we refer the reader to \cite{AC}, \cite{AC2}, or \cite{ARS}.

Next we define \textbf{string algebras}, which were introduced in \cite{BR} and will be considered throughout here:

\begin{defi}
\label{def:string}
    We say that a finite dimensional $k$-algebra $A$  is a \textbf{string algebra} if $A$ is isomorphic to a bound path algebra $kQ/I$ satisfying:

    \begin{enumerate}
        \item For every vertex $x \in Q_0$, there are at most two arrows starting and at most two arrows ending at $x$.

        \item Given an arrow $\alpha \in Q_1$ there is at most one arrow $\beta \in Q_1$ such that $\beta \alpha \notin I$ and at most one arrow $\gamma$ such that $\alpha \gamma \notin I$.

        \item The ideal $I$ is generated by a set of paths over $Q$ (that means, $A$ is a \textit{monomial} algebra).
    \end{enumerate}
\end{defi}

\begin{defi}
\label{def:string peak deep}
    Let $A = kQ/I$ be a string algebra.

    \begin{itemize}
        \item A \textbf{string} over $(Q,I)$ is either a zero-length path $\epsilon_{x}$ associated to a vertex $x$ of $Q$, or a reduced walk $C = c_n \ldots c_1$ over $Q$, such that no subwalk of $C$ or $C^{-1}$ belongs to $I$. 

        \item We say that a string $C$ \textbf{starts in a peak} if there is no arrow $\beta$ of $Q$ such that $C\beta$ is a string. We say that a string $C$ \textbf{starts in a deep} if there is no arrow $\gamma$ of $Q$ such that $C\gamma^{-1}$ is a string.

        \item Dually, we say that a string $C$ \textbf{ends in a peak} if the inverse string $C^{-1}$ starts in a peak and we say that a string $C$ \textbf{ends in a deep} if the inverse string $C^{-1}$ starts in a deep.
    \end{itemize}
\end{defi}

Besides defining string algebras, \cite{BR} also defines, for every string $C$ over a string algebra, a \textbf{string module} $M(C)$ associated to $C$, and then proves that every indecomposable module over a string algebra are either string modules or what the authors also introduce as \textit{band modules}. Additionally, the description of irreducible morphisms and Auslander-Reiten sequences between indecomposable modules over string algebras is fully known from \cite{BR}, pages 168-174. Unfortunately, we cannot include the full description here since it would require some significant space, but we will refer back to \cite{BR} for the details we need. 

\section{The main theorem}
\label{sec:main theorem}

Before proving our main theorem, we shall give an example which will motivate the definition of the family of algebras we shall consider.

\begin{ex}
Let $A$ be the path algebra given by the quiver

\begin{displaymath}
    \xymatrix{3 & 2 \ar[l]_{\beta_2}& 1 \ar[l]_{\beta_1} \ar@(ur,ul)[]_{\alpha}& 4 \ar[l]_{\gamma_1}& 5 \ar[l]_{\gamma_2}}
\end{displaymath}

bound by the relations $\alpha^2 = \beta_1 \gamma_1 = \beta_2 \beta_1 = 0$. The Auslander-Reiten quiver of $A$ (here displayed vertically due to space) is given by:

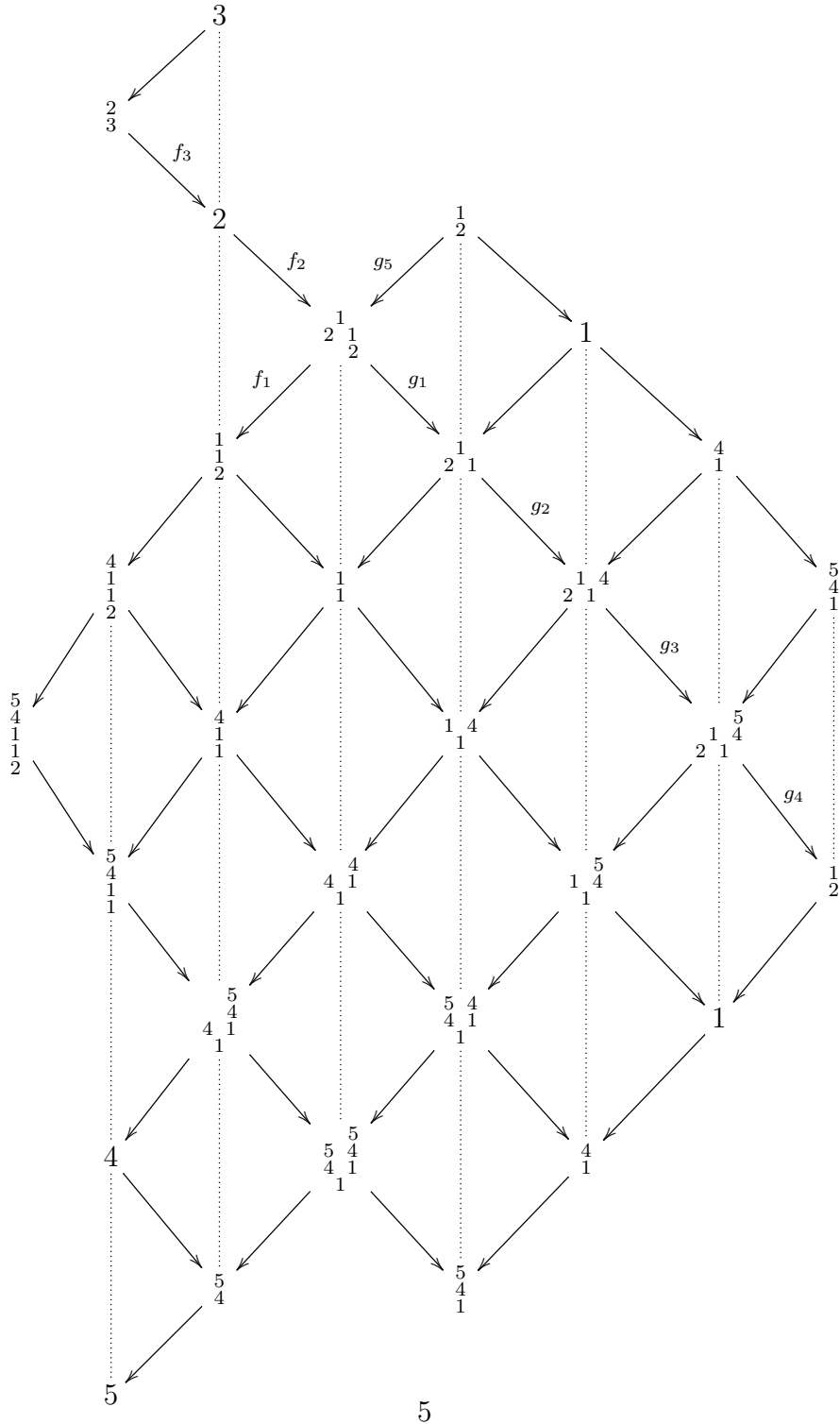
\begin{displaymath}
    \xymatrix{&& 3 \ar@{.}[dd] \ar[dl] & &&&&\\
    & {\begin{smallmatrix}
        2 \\ 3
    \end{smallmatrix}} \ar[dr]^{f_3} &&&&&& \\
    && 2 \ar@{.}[dd] \ar[dr]^{f_2} && {\begin{smallmatrix}
    1 \\ 2  
    \end{smallmatrix}} \ar@{.}[dd] \ar[dr]  \ar[dl]_{g_5} &&& \\
    &&& {\begin{smallmatrix}
      1 \\ 2 \hspace{1ex} 1 \\ \hspace{2ex} 2  
    \end{smallmatrix}} \ar@{.}[dd] \ar[dr]^{g_1}  \ar[dl]_{f_1} && 1  \ar@{.}[dd] \ar[dr] \ar[dl] && \\
    && {\begin{smallmatrix}
      1 \\ 1 \\ 2  
    \end{smallmatrix}} \ar@{.}[dd] \ar[dr] \ar[dl]&& {\begin{smallmatrix}
      1 \\ 2 \hspace{1ex} 1 
    \end{smallmatrix}} \ar@{.}[dd] \ar[dr]^{g_2} \ar[dl]&& {\begin{smallmatrix}
    4 \\ 1 
    \end{smallmatrix}} \ar@{.}[dd] \ar[dr] \ar[dl] & \\
    & {\begin{smallmatrix}
     4 \\ 1 \\ 1 \\ 2  
    \end{smallmatrix}} \ar@{.}[dd] \ar[dr] \ar[dl]&& {\begin{smallmatrix}
    1 \\ 1 
    \end{smallmatrix}} \ar@{.}[dd] \ar[dr]  \ar[dl]&& {\begin{smallmatrix}
      \hspace{1ex} 1 \hspace{1ex} 4 \\ 2 \hspace{1ex} 1 \hspace{1ex}
    \end{smallmatrix}} \ar@{.}[dd] \ar[dr]^{g_3}  \ar[dl]&& {\begin{smallmatrix}
    5 \\ 4 \\ 1  
    \end{smallmatrix}} \ar@{.}[dd] \ar[dl] \\
     {\begin{smallmatrix}
     5 \\ 4 \\ 1 \\ 1 \\ 2  
    \end{smallmatrix}} \ar[dr] && {\begin{smallmatrix}
    4 \\ 1 \\ 1
    \end{smallmatrix}} \ar@{.}[dd] \ar[dr] \ar[dl]&& {\begin{smallmatrix}
    1 \hspace{1ex} 4 \\ 1 
    \end{smallmatrix}} \ar@{.}[dd] \ar[dr] \ar[dl] && {\begin{smallmatrix}
      \hspace{3ex} 5 \\ \hspace{1ex} 1 \hspace{1ex} 4 \\ 2 \hspace{1ex} 1 \hspace{1ex}
    \end{smallmatrix}} \ar@{.}[dd] \ar[dr] ^{g_4}\ar[dl] & \\
    & {\begin{smallmatrix}
    5 \\ 4 \\ 1 \\ 1
    \end{smallmatrix}} \ar@{.}[dd] \ar[dr]  && {\begin{smallmatrix}
    \hspace{2ex} 4 \\ 4 \hspace{1ex} 1 \\ 1 
    \end{smallmatrix}} \ar@{.}[dd] \ar[dr] \ar[dl]&& {\begin{smallmatrix}
    \hspace{2ex} 5 \\ 1 \hspace{1ex} 4 \\ 1 
    \end{smallmatrix}} \ar@{.}[dd] \ar[dr] \ar[dl]&& {\begin{smallmatrix}
    1 \\ 2  
    \end{smallmatrix}} \ar[dl] \\
    &&{\begin{smallmatrix}
    \hspace{2ex} 5 \\ \hspace{2ex} 4 \\ 4 \hspace{1ex} 1 \\ 1 
    \end{smallmatrix}} \ar@{.}[dd] \ar[dr] \ar[dl]&& {\begin{smallmatrix}
    5 \hspace{1ex} 4 \\  4 \hspace{1ex} 1 \\ 1 
    \end{smallmatrix}} \ar@{.}[dd] \ar[dr] \ar[dl]&& 1 \ar[dl] & \\
    & 4 \ar@{.}[dd] \ar[dr] && {\begin{smallmatrix}
    \hspace{2ex} 5 \\ 5 \hspace{1ex} 4 \\ 4 \hspace{1ex} 1 \\ 1 
    \end{smallmatrix}} \ar[dr] \ar[dl] && {\begin{smallmatrix}
    4 \\ 1  
    \end{smallmatrix}} \ar[dl] && \\
    && {\begin{smallmatrix}
    5 \\ 4  
    \end{smallmatrix}} \ar[dl] && {\begin{smallmatrix}
    5 \\ 4 \\ 1  
    \end{smallmatrix}} &&& \\
    & 5 &&&&&&
    }
\end{displaymath}

In the quiver above, there are four modules, namely $\begin{smallmatrix}
    1 \\ 2
\end{smallmatrix},1, \begin{smallmatrix}
    4 \\ 1
\end{smallmatrix}$ e $\begin{smallmatrix}
    5 \\ 4 \\ 1
\end{smallmatrix}$, that have been written twice: their respective two copies should be identified. (It is a quiver `immersible' in a Möbius strip). Dotted lines represent Auslander-Reiten translations. The morphisms $f_1,f_2,f_3,g_1,g_2,g_3,g_4,g_5$ are all canonical inclusions or projections, and they are all irreducible. (Here $f_2$ and $g_5$ being canonical inclusions means that $2 \oplus \begin{smallmatrix}
    1 \\ 2
\end{smallmatrix} \xrightarrow{(f_2 \hspace{1ex} g_5)} \begin{smallmatrix}
    1 \\ 2 \hspace{1ex} 1 \\ \hspace{2ex} 2
\end{smallmatrix}$ is the canonical inclusion of the radical of the projective module $\begin{smallmatrix}
    1 \\ 2 \hspace{1ex} 1 \\ \hspace{2ex} 2
\end{smallmatrix}$ as a submodule).

Note that $f_1 f_2 f_3 = 0$ and that $f_1 g_5 g_4 g_3 g_2 g_1 f_2 f_3 \neq 0$. Actually, due to the Igusa-Todorov theorem, we know that $f_1 g_5 g_4 g_3 g_2 g_1 f_2 f_3 \in \rad^8 \setminus \rad^9$. Therefore, if we define $f_1' = f_1 + f_1 g_5 g_4 g_3 g_2 g_1$, then $f_1'$ is irreducible and $f_1' f_2 f_3 = f_1 g_5 g_4 g_3 g_2 g_1 f_2 f_3 \in \rad^8 \setminus \rad^9$. This is, therefore, an example of 3 irreducible morphisms between indecomposable modules whose composition belongs to $\rad^8 \setminus \rad^9$.

\end{ex}

Having concluded the example above, we proceed to the definition of the family of algebras for our main theorem. For $n \geq 2$ and $m \geq 0$, we define the path algebra $A(n,m)$ given by the quiver

\begin{displaymath}
    \xymatrix{n & \ldots \ar[l]_{\beta_{n-1}} & 2 \ar[l]_{\beta_2}& 1 \ar[l]_{\beta_1} \ar@(ur,ul)[]_{\alpha}& n+1 \ar[l]_{\gamma_1}& n+2 \ar[l]_{\gamma_2}& \ldots \ar[l]& n+m \ar[l]_{\gamma_m}}
\end{displaymath}

and bound by $\alpha^2 = \beta_1 \gamma_1 = \beta_2 \beta_1 = 0$. (If $n = 2$, the relation $\beta_2 \beta_1 = 0$ does not apply).

That way, the algebra from the example above coincides with the algebra $A(3,2)$.

\begin{obs}
Note that each $A(n,m)$ fulfills the conditions in Definition~\ref{def:string}, and therefore it is a string algebra, and we can use the results from \cite{BR}. In particular, one can use them to verify that each $A(n,m)$ has finite type.
\end{obs}

Having defined the family of algebras above, we have a couple of lemmas that will help with the proof of our main theorem:

\begin{lem}
\label{lem:s2 p1 t-1s2}
    For $n \geq 2$ and $m \geq 0$, there is an Auslander-Reiten sequence between $A(n,m)$-modules of the form:

    $$0 \rightarrow S(2) \rightarrow P(1) \rightarrow \tau^{-1} S(2) \rightarrow 0$$

    (which, in particular, has indecomposable middle term).
\end{lem}

\begin{proof}
    Starting from arrow $\beta_1$ and composing with inverses of arrows, we arrive at the string $C = \beta_1 \alpha^{-1} \beta_1^{-1}$, where $C$ begins in a deep and ends in a peak. Then \cite{BR}, Corollary on p. 174 gives us the following Auslander-Reiten sequence:

        $$0 \rightarrow M(\epsilon_2) \xrightarrow{f_2} M(\beta_1 \alpha^{-1} \beta_1^{-1}) \xrightarrow{f_1} M(\alpha^{-1} \beta_1^{-1}) \rightarrow 0$$

        where $f_2$ is a canonical inclusion and $f_1$ is a canonical projection. Now it remains to see that $M(\epsilon_2) = S(2)$ and $M(\beta_1 \alpha^{-1} \beta_1^{-1}) = P(1)$. 
\end{proof}

\begin{lem}
\label{lem:sectional path}
    For $n \geq 2$ and $m \geq 0$, there is a sectional path between $A(n,m)$-modules of the form

     \begin{align*}
     I(n) \xrightarrow{f_n}  \ldots \xrightarrow{f_4} I(3) \xrightarrow{f_3} S(2) \xrightarrow{f_2} &\\
    P(1) \xrightarrow{g_1} \tau^{-1} M(\beta_1) \xrightarrow{g_2} & M(\beta_1 \alpha^{-1} \gamma_1) \xrightarrow{g_3} M(\beta_1 \alpha^{-1} \gamma_1 \gamma_2) \xrightarrow{g_4} \\
    \ldots \xrightarrow{g_{m+1}} &M(\beta_1 \alpha^{-1} \gamma_1 \gamma_2) \xrightarrow{g_{m+2}} M(\beta_1) \xrightarrow{g_{m+3}}P(1) \xrightarrow{f_1} \tau^{-1} S(2)
     \end{align*}

     where the $f_i$'s and $g_i$'s are all canonical inclusions or projections.
\end{lem}

\begin{proof}
    We shall divide the proof of this lemma through the facts stated below:

        \textit{Fact 1: We have a sectional path $I(n) \xrightarrow{f_n} \ldots \xrightarrow{f_4} I(3) \xrightarrow{f_3} S(2)$, where $f_n,\ldots,f_4,f_3$ are canonical projections:}

        Note that $I(3) = M(\beta_2)$, $I(4) = M(\beta_3 \beta_2)$, and so on until we have $I(n) = M(\beta_{n-1} \ldots \beta_3 \beta_2)$. Thus by the description of irreducible modules in \cite{BR}, we have the following irreducible morphisms given by canonical projections:

        \begin{align*}
            I(n) = M(\beta_{n-1} \beta_{n-2} \ldots \beta_3 \beta_2) &\xrightarrow{f_n} M(\beta_{n-2} \ldots \beta_3 \beta_2) = I(n-1) \\
            & \vdots \\
            I(4) = M(\beta_3 \beta_2) &\xrightarrow{f_4} M(\beta_2) = I(3) \\
            I(3) = M(\beta_2) &\xrightarrow{f_3} M(\epsilon_2) = S(2) \\
        \end{align*}

        Moreover, the path $f_n \ldots f_4 f_3$ is sectional because $I(n),\ldots,I(3)$ are injective modules.

        \textit{Fact 2: We have a sectional path $P(1) = M(\beta_1 \alpha \beta_1^{-1}) \xrightarrow{g_1} \tau^{-1} M(\beta_1) = M(\beta_1 \alpha^{-1}) \xrightarrow{g_2} M(\beta_1 \alpha^{-1} \gamma_1) \xrightarrow{g_3} M(\beta_1 \alpha^{-1} \gamma_1 \gamma_2) \xrightarrow{g_4} \ldots \xrightarrow{g_{m+1}} M(\beta_1 \alpha^{-1} \gamma_1 \gamma_2 \ldots \gamma_m) \xrightarrow{g_{m+2}} M(\beta_1) \xrightarrow{g_{m+3}} P(1) = M(\beta_1 \alpha \beta_1^{-1})$, where $g_1,g_2,\ldots,g_{m+3}$ are canonical inclusions or projections:}

        We have that $\beta_1$ does not start in a peak, but ends in one. That way, using notations from \cite{BR}, we have that $\beta_1 = {}_c(\epsilon_1)$ and that $(\beta_1)_h = \beta_1 \alpha \beta_1^{-1}$, where $(\beta_1)_h$ begins in a deep. Then \cite{BR}, Proposition on p. 172 gives us the following Auslander-Reiten sequence starting in $M(\beta_1)$:

        $$0 \rightarrow M(\beta_1) \xrightarrow{\begin{pmatrix}
             p & g_{m+3}
         \end{pmatrix}} M(\epsilon_1) \bigoplus M(\beta_1 \alpha \beta_1^{-1}) \xrightarrow{\begin{pmatrix}
             i \\ -g_1
         \end{pmatrix}} M(\beta_1 \alpha^{-1}) \rightarrow 0$$

         where $i, g_{m+3}$ are canonical inclusions and $p, g_1$ are canonical projections.

         We have that $\beta_1 \alpha^{-1}$ does not start in a peak. Since $(\beta_1 \alpha^{-1})_h = \beta_1 \alpha^{-1} \gamma_1$, we have a canonical inclusion $g_2: M(\beta_1 \alpha^{-1}) \hookrightarrow M(\beta_1 \alpha^{-1} \gamma_1)$, which is irreducible.

         In the same manner, $\beta_1 \alpha^{-1} \gamma_1$ does not start in a peak. We have that $(\beta_1 \alpha^{-1} \gamma_1)_h = \beta_1 \alpha^{-1} \gamma_1 \gamma_2$, and thus we have a canonical inclusion $g_3: M(\beta_1 \alpha^{-1} \gamma_1) \hookrightarrow M(\beta_1 \alpha^{-1} \gamma_1 \gamma_2)$, which is irreducible.

         And so on we can construct the morphisms $g_4,\ldots,g_{m+1}$. Now, since $\beta_1$ does not start in a deep and $(\beta_1)_c = \beta_1 \alpha^{-1} \gamma_1 \gamma_2 \ldots \gamma_m$, we have a canonical projection $g_{m+2}: M(\beta_1 \alpha^{-1} \gamma_1 \gamma_2 \ldots \gamma_m) \twoheadrightarrow M(\beta_1)$, which is irreducible.

         It remains to see that $g_1,\ldots,g_{m+3}$ is a sectional path. In order to do that we apply the Proposition on p. 172 of \cite{BR} at the sub-items below:
         
         \begin{itemize}
             \item We have that $\beta_1 \alpha \beta_1$ starts and ends in a peak, with $\beta_1 \alpha \beta_1 = {}_c(\alpha)_c$. Thus $\tau^{-1} M(\beta_1 \alpha^{-1} \beta_1) = M(\alpha) \neq M(\beta_1 \alpha^{-1} \gamma_1)$.

             \item We have that $\beta_1 \alpha^{-1}$ does not start in a peak, but ends in one, with $\beta_1 \alpha^{-1} = {}_c(\alpha^{-1})$. Thus $\tau^{-1} M(\beta_1 \alpha^{-1}) = M((\alpha^{-1})_h) = M(\alpha^{-1} \gamma_1) \neq M(\beta_1 \alpha^{-1} \gamma_1 \gamma_2)$.

             \item We have that $\beta_1 \alpha^{-1} \gamma_1$ does not start in a peak, but ends in one, with $\beta_1 \alpha^{-1} \gamma_1 = {}_c(\alpha^{-1} \gamma_1)$. Thus $\tau^{-1} M(\beta_1 \alpha^{-1} \gamma_1) = M((\alpha^{-1} \gamma_1)_h) = M(\alpha^{-1} \gamma_1 \gamma_2) \neq M(\beta_1 \alpha^{-1} \gamma_1 \gamma_2 \gamma_3)$.

             \item And so on we can conclude that $\tau^{-1} M(\beta_1 \alpha^{-1} \gamma_1 \ldots \gamma_l) = M( \alpha^{-1} \gamma_1 \ldots \gamma_{l+1})$ for every $l$ between 2 and $m-1$.

             \item Also we may not have something like $\tau^{-1} M(\beta_1 \alpha^{-1} \gamma_1 \ldots \gamma_l) = M(\beta_1 \alpha \beta_1^{-1})$, because $M(\beta_1 \alpha \beta_1^{-1})$ is projective.
         \end{itemize}

         That concludes the proof of Fact 2.

        \textit{Fact 3: The path $I(n) \xrightarrow{f_n} \ldots \xrightarrow{f_4} I(3) \xrightarrow{f_3} S(2) \xrightarrow{f_2} P(1) \xrightarrow{g_1} \tau^{-1} M(\beta_1) \xrightarrow{g_2} M(\beta_1 \alpha^{-1} \gamma_1) \xrightarrow{g_3} M(\beta_1 \alpha^{-1} \gamma_1 \gamma_2) \xrightarrow{g_4} \ldots \xrightarrow{g_{m+1}} M(\beta_1 \alpha^{-1} \gamma_1 \gamma_2) \xrightarrow{g_{m+2}} M(\beta_1) \xrightarrow{g_{m+3}} P(1) \xrightarrow{f_1} \tau^{-1} S(2)$ is sectional:}

        Having proved Facts 1 and 2, it only remains to see that $M(\beta_1) \neq S(2)$, which means $S(2) \neq \tau\tau^{-1} M(\beta_1)$ and $M(\beta_1) \neq \tau \tau^{-1} S(2)$.

\end{proof}

\begin{proof}[Proof of the Main Theorem]

    Observe that, using Lemmas~\ref{lem:s2 p1 t-1s2} and~\ref{lem:sectional path}, we have the following configuration in the Auslander-Reiten quiver of $A(n,m)$:
    
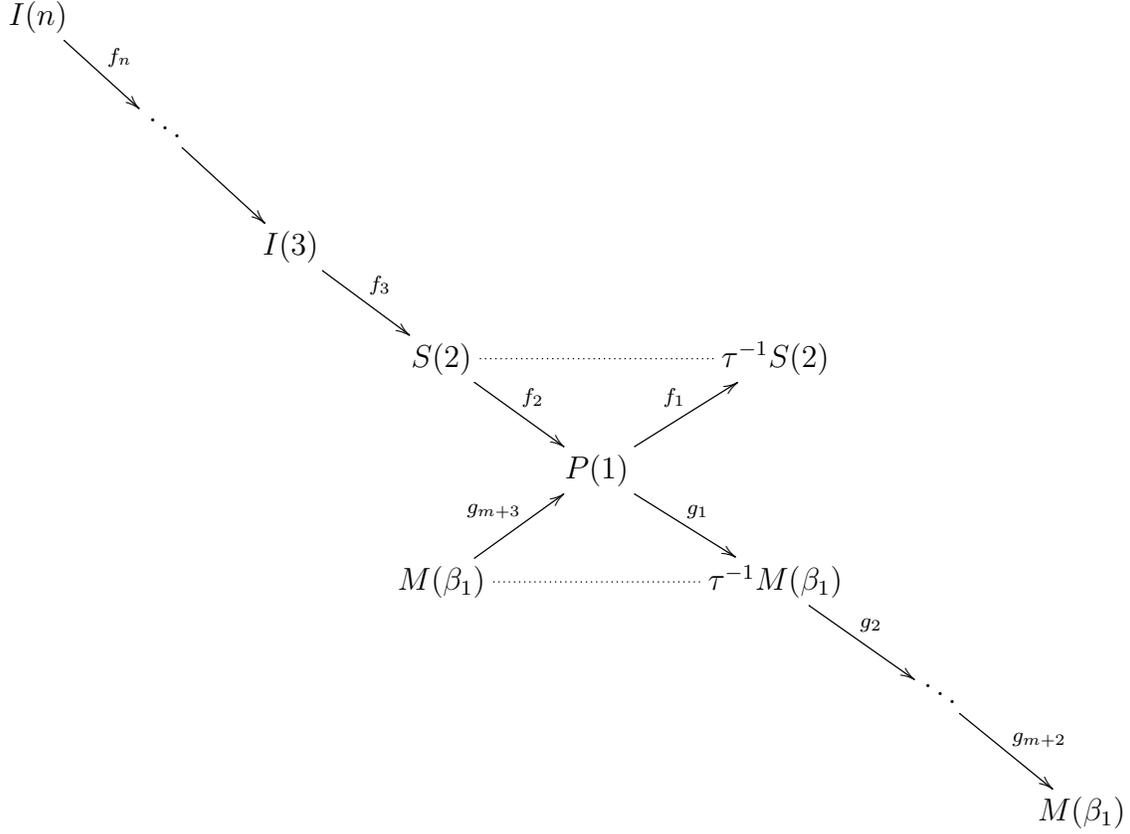
   \begin{displaymath}
    \xymatrix{
        I(n) \ar[dr]^{f_n} &&&&&&& \\
        & \ddots \ar[dr]&&&&&& \\
        && I(3) \ar[dr]^{f_3}&&&&& \\
        &&& S(2) \ar[dr]^{f_2} \ar@{.}[rr] && \tau^{-1} S(2) && \\
        &&&& P(1) \ar[ur]^{f_1} \ar[dr]^{g_1} &&& \\
        &&& M(\beta_1) \ar@{.}[rr] \ar[ur]^{g_{m+3}}&& \tau^{-1} M(\beta_1) \ar[dr]^{g_2}&& \\
        &&&&&& \ddots \ar[dr]^{g_{m+2}}& \\
        &&&&&&& M(\beta_1)}
    \end{displaymath}

where $f_1,\ldots,f_n,g_1,\ldots,g_{m+3}$ are morphisms satisfying the properties described in Lemma~\ref{lem:sectional path}, and $S(2) \xrightarrow{f_2} P(1) \xrightarrow{f_1} \tau^{-1} S(2)$ actually forms an Auslander-Reiten sequence with indecomposable middle term, which is $P(1)$. From that we have $f_1 f_2 = 0$ e and thus $f_1 f_2 \ldots f_n = 0$. Using the Igusa-Todorov Theorem (\cite{IT1}, \S13), we obtain that $f_1 g_{m+3} \ldots g_1 f_2 \ldots f_n \in \rad^{n+m+3} \setminus \rad^{n+m+4}$.  Then, if we take $f_1' = f_1 + f_1 g_{m+3} \ldots g_1$, we have that $f_1'$ is irreducible and therefore that $f_1' f_2 \ldots f_n$ is a composite of $n$ irreducible morphisms in $\rad^{n+m+3} \setminus \rad^{n+m+4}$, as we wanted to prove.
\end{proof}

\section*{Acknowledgements}

This work is part of the PhD thesis of the first named author, developed under supervision by the second named author. The authors gratefully acknowledge financial support by São Paulo Research Foundation (FAPESP), grants \#2020/13925-6 and \#2022/02403-4 and by CNPq (grant Pq 312590/2020-2).

\end{document}